\documentclass{amsart}

\title[Ranks of mapping tori via the curve complex]{Ranks of mapping tori via the curve complex}

\author{Ian Biringer and Juan Souto}

\usepackage{amsfonts}
\usepackage{amsmath}
\usepackage{amsthm,amstext,amsfonts,bm,amssymb}
\usepackage{amssymb}
\usepackage{xfrac}
\usepackage{refcount}
\usepackage[margin=1cm,singlelinecheck=off]{caption}
\usepackage{enumitem}
\usepackage{graphicx}%
\usepackage{xypic}
\providecommand{\U}[1]{\protect\rule{.1in}{.1in}}

\newcommand{\CE}{\mathcal E}		

		\newcommand{\CN}{\mathcal N}

\theoremstyle{plain}
\newtheorem{theorem}{Theorem}[section]

\newtheorem*{namedtheorem}{\theoremname}
\newcommand{\theoremname}{testing}

\makeatletter
\newtheorem*{rep@theorem}{\rep@title}
\newcommand{\newreptheorem}[2]{%
\newenvironment{rep#1}[1]{%
 \def\rep@title{#2 \ref{##1}}%
 \begin{rep@theorem}}%
 {\end{rep@theorem}}}
\makeatother

\newreptheorem{theorem}{Theorem}
\newreptheorem{lemma}{Lemma}
\newreptheorem{proposition}{Proposition}
\newtheorem{claim}[theorem]{Claim}

\newtheorem{corollary}[theorem]{Corollary}

\newtheorem{lemma}[theorem]{Lemma}
\newtheorem{fact}[theorem]{Fact}
\newtheorem*{lemma*}{Lemma}

\newtheorem{proposition}[theorem]{Proposition}
\newtheorem{remark}[theorem]{Remark}

\theoremstyle{definition}
\newtheorem{definition}[theorem]{Definition}
\newtheorem*{definition*}{Definition}

\newcommand{\length}{\mathrm{length}}
\newcommand{\BZ}{\mathbb{Z}}
\newcommand{\BR}{\mathbb{R}}

\newcommand{\BH}{\mathbb{H}}
\newcommand{\rank}{\mathrm{rank}}
\newcommand{\diameter}{\mathrm{diam}}
\newcommand{\circumference}{\mathrm{circ}}
\newcommand{\into}{\hookrightarrow}
\newcommand{\isometries}{\mathrm{Isom}}

\begin{document}
\maketitle
\begin{abstract}
We show that if $\phi$ is a homeomorphism of a closed, orientable surface of genus $g$, and $\phi$ has large translation distance in the curve complex, then the fundamental group of the mapping torus $M_\phi$ has rank $2g+1$.
\end{abstract}

\section{Introduction}
\label {introduction}

Let $\phi : S \to S $ be a homeomorphism of a closed, orientable surface $S$ of genus $g$.  The \emph {mapping torus} of $\phi $ is the $3 $-manifold $M_\phi$ defined by
$$M_\phi = S \times [0, 1] \, / \sim, \ \ (x,0)\sim(\phi(x),1).$$

Our interest is in the minimal number of elements needed to generate the fundamental group of $M_\phi $; this is called the \emph {rank} of $\pi_1 M_\phi$.  
As we have
$$1 \longrightarrow \pi_1 S \longrightarrow \pi_1 M_\phi \longrightarrow \BZ \longrightarrow 1,$$
a generating set for $\pi_1 M_\phi$ can be constructed by adding to a generating set for $\pi_1 S $ any element of $\pi_1 M_\phi$ the projects to $1\in \BZ$.  Therefore, 
$$\rank(\pi_1 M_\phi) \leq \rank(\pi_1 S) + 1 = 2g+1.$$
The mapping torus of the identity map is $S \times S^1$, in which case the above is  an equality.   However, there are also examples where the inequality is strict.  For instance, if $g=1$ then $\rank (\pi_1 M_\phi)=2 < 3$ when $\phi $ is the map
$$\phi : T^2 \longrightarrow T^2, \ \ \ \phi = \begin {pmatrix} 0 & 1 \\ 1 & 0 \end {pmatrix}. $$ For although $\pi_1 M_\phi = \BZ^2 \rtimes_\phi \BZ$ is generated by $(1,0,0)$, $(0,1,0)$ and $(0,0,1)$, the second is the conjugate of the first by the third.   Similar examples can also be constructed by hand in higher genus.

For a more dramatic example, suppose $M_\phi$ is hyperbolic and $b_1(M_\phi) \geq 2$.  By work of Thurston \cite{Thurstonnorm}, there are surfaces $S'$ with arbitrarily large genus and homeomorphisms $\phi' : S' \to S'$ such that $M_\phi \cong M_{\phi'}$.  So, if the genus $g'$ of $S'$ is large, $\rank(\pi_1 M_{\phi'})=\rank(\pi_1 M_\phi)$ will be much less than $2g'+1$.

\begin {theorem}\label {main}
When $g\geq 2 $, there is some $L = L (g) $ such that if $\phi : S \longrightarrow S$ has translation length at least $L $ in the curve complex $C(S)$, then $$\rank (\pi_1 M_\phi) = 2g+1 $$
and all $2g+1$-element generating sets for $\pi_1 M_\phi $ are Nielsen equivalent.
\end {theorem}

The \emph {curve complex} $C(S)$ is the graph whose vertices are isotopy classes of essential simple closed curves on $S$ and whose edges connect vertices with disjoint representatives.  $C(S)$ has a path metric in which all edges have length $1$, and the \emph {translation length} of $\phi$ is defined by $$\tau(\phi)=\inf_{v\in C(S)} d(v,\phi(v)).$$

Two $n $-element generating sets for a group $G $ are \emph {Nielsen equivalent} if the associated surjections $F_n \longrightarrow G$ differ by precomposition with a free group automorphism. As $Aut(F_n)$ is generated by Nielsen transformations, this is equivalent to requiring that our two generating sets for $G $ are related by a sequence of the following moves:
\begin {enumerate}
\item if $a\neq b$ are generators, replace $a$ with $ab$,
\item if $a\neq b$ are generators, switch their places in the ordering.
\item if $a$ is a generator, replace it with $a^{-1}$.
\end {enumerate}

Note that the choice to replace $a$ with $ab$ rather than $ba$ does not matter, since the two options are conjugate by an inversion (move 3).

\vspace{2mm}

Theorem \ref {main} is a strengthening of previous results of the authors \cite {Biringergeometry}, \cite{Soutorank}.  The proof outline is similar to the argument given in \cite {Biringergeometry}, except that more machinery is required to deal with the thin parts of the mapping tori, which were excluded in the previous work.  We would like to stress, however, that this paper is the first of \cite{Biringergeometry,Biringerfiniteness,Soutorank} to harness the power of carrier graphs when the injectivity radius is not bounded away from zero.

The \emph {Heegaard genus} of a closed, orientable $3$-manifold $M$ is the minimal genus of a \emph{Heegaard splitting}, a surface $S\subset M$ that divides $M $ into two handlebodies.  Rank is always at most Heegaard genus, since the fundamental group of each handlebody surjects onto $\pi_1 M$, and Waldhausen conjectured that they are always equal.  Although this is false \cite{Boileauheegaard,Lirank,Schultensgeometric}, rank and genus still appear to be very related; for instance, in all known examples the ratio is bounded.  

Bachman-Schleimer~\cite{Bachmansurface} showed in 2005 that when $\tau(\phi)\geq 4g$, the Heegaard genus of $M_\phi$ is $2g+1$ and there is a unique minimal genus splitting. So, our theorems combine to say that rank is equal to genus for mapping tori $M_\phi$ with $\tau(\phi)$ large. Moreover, since minimal generating sets and splittings are unique in both cases, this says that every minimal size generating set for $M_\phi$ is \emph {geometric}, i.e.\ it can be realized by a graph in $M_\phi$ that thickens to a Heegaard splitting.
Note that unlike Bachman-Schleimer, we lack an explicit bound for the translation length. While some bookkeeping would make our constant explicit, the resulting bound would be large and uninteresting.

Theorem \ref{main} is still interesting when $S=T^2$. Here, homotopy classes of homeomorphisms $\phi : T^2 \longrightarrow T ^ 2$ correspond to matrices $A \in GL_2\BZ $, and
$$\pi_1 M_\phi=\BZ ^ 2\rtimes_A\BZ . $$
Using the Cayley-Hamilton theorem, one can show $\BZ^2 \rtimes_A \BZ$ has rank $2$ if and only if there is some $v\in \BZ^2$ such that $\langle v,A  v\rangle = \BZ^2$, see e.g.\ \cite[Lemma 3.1]{Levittrank}. 
The curve complex $C(T^2)$ is the \emph {Farey graph}, whose vertices are primitive elements of $\BZ ^ 2 $ up to negation and whose edges connect generating pairs.  So, if the translation distance $\tau(A)$ of $A \circlearrowright C(T^2)$ is at least $2 $, then $\rank(G_A)=3.$ Note that the analogous statement about Heegard genus  is an earlier theorem of Cooper-Scharlemann~\cite{Cooperstructure}.

The theorem is still true when $S$ has punctures, i.e.\ $\rank (\pi_1 M\phi)=b_1(M_\phi)+1$ when $\phi$ has large translation distance, and the argument goes through with minimal changes.  In Lemma \ref{thickcircumference}, only Margulis tubes must be electrified, while the cusp neighborhoods are ignored. Instead of using Zieschang's theorem \cite {Zieschanguber} on Nielsen equivalence of generating sets for $\pi_1 S$, one uses the fact that minimal size generating sets for a free group are free bases. Some of the work in \S \ref{lemmasection} becomes more complicated, but this is mostly bookkeeping.

This paper was partly motivated by a conversation with Igor Rivin, who asked us whether a `random' hyperbolic mapping torus has the expected rank.  Maher \cite{Maherlinear} has shown that certain random walks in the mapping class group of a surface $S$ make linear progress in the complex of curves.  So, combining Maher's full statement with our theorem gives:

\begin {corollary}\label {random}
Suppose that $S$ is a closed, orientable surface with genus $g\geq 2$ and let $\mu$ be a probability distribution on $\mathrm{Mod}(S)$ with finite first moment, whose support is bounded in the relative metric and generates a non-elementary subgroup.  If $\omega_n$ is a $\mu$-random walk on $\mathrm{Mod}(S)$, then the probability that $\rank (\pi_1 M_{\omega_n}) = 2g+1$ converges to $1$ as $n\to \infty$.

\end {corollary}

The `relative metric' on $\mathrm{Mod}(S)$ is a metric under which the mapping class group is quasi-isometric to the complex curves, and the condition of bounded support is equivalent to requiring that all elements of the support translate a given base point in the curve complex by a bounded amount. We refer the reader to Maher's paper \cite{Maherlinear} for more details.  In particular, though, a simple random walk on a Cayley graph for $\mathrm{Mod}(S)$ satisfies the assumptions above.

We refer the reader to Rivin's article \cite {Rivinstatistics} for a number of similar analyses of invariants of random hyperbolic $3$-manifolds, both in the fibered context as above and in the random Heegaard splitting model of Dunfield-Thurston \cite{Dunfieldfinite}.

\subsection*{Organization of the paper}

Section \ref{proof} contains a top-level proof of Theorem \ref {main}. Many of the important ingredients in the proof are stated here, with their proofs deferred to later sections. Section \ref{preliminaries} presents the necessary background, which includes convex cores, ends, and simplicial hyperbolic surfaces. Section \ref{circumferencesection} relates the `electric circumference' of a mapping torus $M_\phi$ to the translation distance $\tau(\phi)$. This statement is implicit in Minsky's program to build bilipschitz models for $3$-manifolds \cite{Brockclassification,Minskybounded} and appears in a paper of Bowditch \cite{Bowditchending}, but we believe that the proof here is easier to read. Section \ref{carriergraphsection} develops the theory of \emph {carrier graphs}, the main technical tool of the paper. Finally, the proof of an important lemma is presented in Section \ref{lemmasection}.

\subsection*{Acknowledgments}
Thanks are due to Igor Rivin for requesting that we write this article, and to Robert Tang for pointing out a reference for Theorem 4.1.  The referee also contributed enormously to the paper, improving both its readability and accuracy. The first author was partially supported by NSF grant DMS-1308678.

 \section{The proof of Theorem \ref {main}}

Let $S $ be a closed, orientable surface with genus $g\geq 2$, fix a homeomorphism $\phi : S \longrightarrow S $ and let $ M_\phi$ be the associated mapping torus.  We will assume in this section that $\phi$ is pseudo-Anosov, since finite order and reducible $\phi$ can be excluded from Theorem \ref {main} by taking $L$ large enough.

 By Thurston's hyperbolization theorem  \cite{Thurstonhyperbolic2}, $M_\phi$ admits a hyperbolic metric.   Fix some small $\epsilon>0$ and let $d_{\epsilon} $ be the path pseudo-metric on $M_\phi $ obtained by   {electrifying} the $\epsilon $-thin part of $M_\phi $ (see \S\ref {elecgeom}).  The \emph {$\epsilon $-electric circumference} of $M_\phi $ is the smallest $d_\epsilon $-diameter of a loop in $M_\phi$ that projects nontrivially to $\pi_1 S ^ 1 $.

\begin {lemma}[Electric circumference $\approx \tau(\phi)$]
\label{thickcircumference}
For small $\epsilon=\epsilon(g) >0 $, there is some constant $k=k(g,\epsilon)$ such that $\frac 1k \cdot \tau (\phi) \,\leq \, \circumference_{ \epsilon} (M_\phi) \leq k \cdot \tau (\phi). $
\end {lemma}

Instead of electric circumference, one could also consider either the $d_\epsilon $-diameter of $M_\phi$ or the smallest $d_\epsilon $-length of a loop projecting nontrivially to $\pi_1 S ^ 1 $.   However, the arguments of \S  \ref {circumferencesection} imply that both of these are equal to electric circumference  up to an error depending on $g,\epsilon $.  We use the definition of electric circumference above since it best fits our needs later in the section.

Lemma \ref {thickcircumference} follows from a theorem of Bowditch \cite[Theorem 5.4]{Bowditchending}, which we state in \S \ref {circumferencesection}.  It is also a moral starting point for Minsky's program \cite {Minskyclassification, Minskykleinian, Minskybounded} for building bilipschitz models for hyperbolic $3$-manifolds.

\vspace{1mm}

The goal now is to show that when the $\epsilon $-electric circumference of $M_\phi $ is large, the fundamental group $\pi_1 M_\phi $ has rank $2g+1$ and has a single Nielsen equivalence class of $2g+1$-element generating sets.  To do this, we need a tool that allows us to view generating sets for $\pi_1 M$ geometrically.

A \emph{carrier graph} in a hyperbolic manifold $M $ is a map $f : X \longrightarrow M$, where $X$ is a connected finite graph, that induces a surjection on $\pi_1 $.  Two carrier graphs $f: X\longrightarrow M $ and $g: Y\longrightarrow M $ are \emph {equivalent} if there is a homotopy equivalence $h:X\longrightarrow Y $ such that $g\circ h$ is homotopic to $f $.
Note that any generating set for $\pi_1 M $ gives, for instance, a carrier graph whose domain is a wedge of circles.  This map induces a bijection between Nielsen equivalence classes of generating sets (see \S \ref {introduction}) and equivalence classes of carrier graphs as above.

  As long as a carrier graph $f$ is rectifiable, the hyperbolic metric on $M$ pulls back to a path (pseudo)-metric on $X$, which we use to measure the lengths of the edges of $X$.  We'll say that $f $ has \emph {minimal length} if its total edge length is at most that of any other equivalent carrier graph $Y \longrightarrow M$. Arzela-Ascoli's theorem implies that any equivalence class of carrier graphs has a minimal length representative.

White \cite {Whiteinjectivity} has shown that the geometry of a minimal length carrier graph is  well-behaved: for instance, $X$ is trivalent and the edges of $X$ map to geodesic segments in $ M$ that connect at $2\pi/3$-angles.  Using these features, he showed that any minimal length carrier graph has a cycle with total edge length less than some constant depending on $\rank(\pi_1 X)$.  An even stronger statement was proved in \cite{Biringergeometry}: any minimal length carrier graph $f :X \to M$ is filtered by subgraphs each of which has bounded length `relative to' the previous subgraph.

\begin{proposition}[Chains of Bounded Length, see \S \ref {relativelengthsection}]
\label{chainsprop}
Fixing $\epsilon>0 $, there is an $L=L(\epsilon,n) $ such that if $M$ is a hyperbolic $n$-manifold such that $\pi_1 M$ does not split is a free product, e.g.\ if $M$ is closed, and $f: X \to M$ is a minimal length carrier graph, there is a sequence of (possibly disconnected) subgraphs $\emptyset = Y_0 \subset Y_1 \subset \ldots \subset Y_k = X$ such that $\length_{Y_i,\epsilon}(Y_{i+1}) < L$ for all $i$.
\end{proposition}

Intuitively, $\length_{Y_i,\epsilon}(Y_{i+1})$ is the hyperbolic length of the part of $f(Y_{i + 1})$ that lies outside of the union of the following sets:
\begin {enumerate}
\item the convex hulls of the components of $f(Y_i)$,
\item any $\epsilon $-thin parts of $M$ where a curve from $f_*(\pi_1 Y_i)$ is short.
\end {enumerate}
However, to make this precise one must work in the universal cover -- see \S \ref {relativelengthsection}.  

\vspace{1mm}

Our primary application of the carrier graph machinery is the following claim, which is the heart of the proof of Theorem \ref{main}.

\begin {claim}  \label {theclaim}Given $n,\epsilon>0 $, there is some $L = L (n,\epsilon) $ as follows.  Suppose that $$f: X\longrightarrow M_\phi $$ is a minimal length carrier graph with $\rank (\pi_1 X) = m $.  If the $\epsilon $-electric circumference of $M_\phi $ is at least $L= L (m,\epsilon) $, then there is a subgraph $Y\subset X $ such that $f_*(\pi_1 Y)$ lies with finite index in the fiber subgroup $\pi_1 S \subset \pi_1 M_\phi $.
\end {claim}
\begin {proof}
Consider the filtration of $X $ from Proposition \ref {chainsprop}: $$\emptyset = Y_0 \subset Y_1 \subset \ldots \subset Y_k = X.$$ 

\begin {lemma}[see \S \ref {lemmasection}]\label{hullbounds}
Fixing $i$, suppose that for each connected component $Y_i ^j\subset Y_i $, the image $f_*(\pi_1 Y_i^ j) $ is an infinite index subgroup of $\pi_1 S \lhd \pi_1 M_\phi $.  

Then for every component $Y^k_{i+1}$ of $ Y_{i + 1}$, the $d_{\epsilon} $-diameter of $f (Y^ k_{i+1})$ in $M_\phi $ is bounded above by a constant depending only on $g,\epsilon $ and the $d_{\epsilon} $-diameters of $f (Y^j_i) $, where $Y_i^j$ is a component of $Y_i $.
\end {lemma}

We will prove the lemma in \S \ref {lemmasection}, but the point is the following.  Since we know that $\length_{Y_i,\epsilon}(Y_{i+1}) $ is bounded, the $d_{\epsilon }$-diameter of the part of each $f(Y^k_{i + 1})$ that lies outside of the convex hulls of the $f(Y^j_i) \subset M_\phi$ is also bounded.  So, the goal is to bound the $d_{\epsilon} $-diameter of these convex hulls.

Applying Lemma \ref {hullbounds} iteratively, we get a bound in terms of $m,\epsilon $ for the $\epsilon$-electric diameter of the first subgraph $Y_i $ with a component $Y=Y_i ^ j $ such that $f_*(\pi_1 Y) $ is \emph {not} an infinite index subgroup of $\pi_1 S \subset \pi_1 M_\phi$.   As long as the $\epsilon $-electric circumference of $M_\phi $ is larger than this diameter bound, the group $f_*(\pi_1 Y) $ still lies in $\pi_1 S $, and therefore is a finite index subgroup of $\pi_1 S $.
\end {proof}

Now that we have proved the claim, we derive Theorem \ref {main}.

\begin {reptheorem}{main}
When $g\geq 2 $, there is some $L = L (g) $ such that if $\phi : S \longrightarrow S$ has translation length at least $L $ in the curve complex $C(S)$, then $$\rank (\pi_1 M_\phi) = 2g+1 $$
and all $2g+1$-element generating sets for $\pi_1 M_\phi $ are Nielsen equivalent.
\end {reptheorem}
\begin {proof}

Fixing $\epsilon=\epsilon (g)$ as in Lemma \ref {thickcircumference}, Claim \ref {theclaim} provides some $L=L(g)$ such that when the $\epsilon $-electric circumference of $M_\phi$ is at least $L $, any minimal length carrier graph $f: X\longrightarrow M_\phi$ with $\rank (\pi_1 X) \leq 2g+1$ has a subgraph $Y\subset X $ such that $f_*(\pi_1 Y)$ lies with finite index in the fiber subgroup $\pi_1 S \subset \pi_1 M_\phi $.  In fact, as $\rank (\pi_1 Y) < \rank (\pi_1 X) \leq 2g + 1 $, we have $f_*(\pi_1 Y)=\pi_1 S $, as any proper finite index subgroup of $\pi_1 S $ has larger rank.

So, suppose the $\epsilon $-electric circumference of $M_\phi$ is at least $L $ and $E $ is a generating set for $\pi_1 M_\phi$ with at most $2 g + 1 $ elements.  By Arzela-Ascoli, there is a minimal length carrier graph $f: X\longrightarrow M_\phi$ as above in the equivalence class corresponding to $E $.  From above, $\pi_1 X$ has a $2g+1$-element free basis such that the images of $2g$ of these elements generate $\pi_1 S \subset \pi_1 M_\phi$. So, $E$ has size $2g+1$ and is Nielsen equivalent to a generating set for $\pi_1 M_\phi$ of the form
$$\{e_1,\ldots,e_{2g+1}\} \subset \pi_1 M_\phi, \ \ \left<e_1,\ldots, e_{2g}\right> =\pi_1 S.$$
Zieschang \cite {Zieschanguber} has shown that any two minimal size generating sets for a surface group are Nielsen equivalent, so it follows immediately that any two generating sets for $\pi_1 M_\phi$ of the form above are also Nielsen equivalent.
\end {proof}

\begin {remark} In the proof of Lemma \ref{hullbounds}, we need to say that some properties of minimal length carrier graphs are inherited by subgraphs.  Instead of giving an argument specific to the situation, we develop all the minimal length machinery for carrier graphs that only have \emph {minimal length rel vertices}, see \S \ref {edgemovesection}, a property that \emph{is} inherited by subgraphs. This more relaxed version of minimality will also play a role in a forthcoming paper of the authors describing the geometry of closed hyperbolic $3$-manifolds with bounded rank and injectivity radius.
\end {remark}

\label {proof}

\section{Preliminaries}
\label {preliminaries}

\subsection {Convex cores and ends of hyperbolic $3$-manifolds} 
\label {cores}
Let $M$ be a hyperbolic $3$-manifold with finitely generated fundamental group, and assume that $M $ has no cusps.  By the Tameness theorem \cite {Agoltameness,Calegarishrinkwrapping}, $M $ is homeomorphic to the interior of a compact $3$-manifold.  If $M $ is noncompact, each of its ends then has a neighborhood homeomorphic to $\Sigma\times (0,\infty) $ for some closed surface $\Sigma $.  In the following, we recall a \emph{geometric} classification of these ends.  As all the stated results are extremely well-known, we eschew pinpoint references and refer the reader to \cite{Bonahonbouts,Canaryends,Matsuzakihyperbolic,Thurstongeometry}, for instance, for details.

The \emph {convex core}  of $M$, written $core(M)$, is the smallest convex submanifold of $M$ whose inclusion is a homotopy equivalence.  If $M=\BH^3 / \Gamma $, where $\Gamma $ is a discrete group of isometries of $\BH ^ 3 $, then  $core(M)$ is the quotient by $\Gamma $ of the convex hull in $\BH ^ 3 $ of the limit set $\Lambda (\Gamma) \subset \partial_\infty\BH ^ 3 $. The ends of $M$ are classified according to their intersections with $core(M)$.  Namely, an end $\CE$ of $M$ either has a neighborhood disjoint from $core(M)$, in which case $\CE$ called \emph {convex-compact}, or has a neighborhood contained in $core(M)$, in which case it is called \emph {degenerate}.

Both types of ends can occur in hyperbolic $3$-manifolds $M\cong S\times\BR $.   For example, orthogonally extending a discrete, cocompact action of $\Gamma =\pi_1 S$ on some totally geodesic plane $\BH^2 \subset \BH^3$ gives an action $\Gamma \circlearrowright \BH ^ 3 $.  The quotient $M=\BH ^ 3/\Gamma$ is homeomorphic to $S\times \BR $ and is called \emph{Fuchsian}.  In this case, the convex core is $\BH ^ 2/\Gamma$, which is compact, so both ends are convex cocompact.

Now suppose $\phi: S\longrightarrow S $ is pseudo-Anosov.  By Thurston's hyperbolization theorem \cite{Thurstonhyperbolic2}, the mapping torus $M_\phi $ admits a hyperbolic metric.  Let $\hat M_\phi $ be the infinite cyclic cover of $M_\phi $ corresponding to the subgroup $\pi_1 S \subset \pi_1 M_\phi $.  Then $\hat M_\phi\cong S\times\BR $, and the hyperbolic metric on $M_\phi $ lifts to a metric on $\hat M_\phi $.  As $\hat M_\phi $ regularly covers a closed manifold, it is its own convex core; for instance, one can see this by noting that a cocompact group $\Gamma $ of isometries of $\BH ^ 3 $ has full limit set and a nontrivial normal subgroup of $\Gamma $ has the same limit set as $\Gamma $.  It follows that both ends of $\hat M_\phi $ are degenerate.

\subsection {Electric geometry and simplicial hyperbolic surfaces}
\label{elecgeom}

We introduce here some terminology related to the geometry of hyperbolic manifolds with electrified thin parts. Afterwards, we describe a useful class of maps of surfaces into hyperbolic $3$-manifolds, called \emph {simplicial hyperbolic surfaces}, about which more information can be found in \cite {Calegarishrinkwrapping,Canarycovering,Somaexistence}.

Suppose that $M $ is a hyperbolic manifold without cusps --- here, the dimension of $M $ will always be either $2$ or $3$. The \emph {injectivity radius} at a point $p\in M$, written $\mathrm{inj}_M (p)$, is half the length of a shortest homotopically essential loop through $p $.  The \emph {$\epsilon $-thin part} of $M $ is the subset
$$M_{\leq \epsilon} = \{ p \in M\ | \ \mathrm {inj}_M (p) \leq\epsilon\}. $$
The complement of $M_{\leq \epsilon} $ is the \emph {$\epsilon $-thick} part of $M $, written $M_{>\epsilon} $.
When $\epsilon $ is less than the Margulis constant (see \cite {Benedettilectures}), every component of $M_{\leq\epsilon} $ is a neighborhood of a closed geodesic in $M $ and is homeomorphic either to  a solid torus (when $\dim(M)=3$) or to an annulus (when $\dim(M)=2$).

We construct a path pseudometric $d_\epsilon $ on $M $ by declaring that the length of any path contained in $M_{\leq\epsilon}$ is zero, while path lengths outside $M_{\leq\epsilon}$ are measured hyperbolically.  We say that $d_\epsilon $ is constructed from the hyperbolic metric on $M$ by \emph {electrifying} the $\epsilon $-thin part of $M $.  Metric quantities such as length and diameter have their \emph {$\epsilon $-electric} counterparts, written e.g.\ $\length_\epsilon $ and $\diameter_\epsilon $.

\begin {definition}A \emph {simplicial hyperbolic surface} is a map $f: S \longrightarrow M$ such that
\begin {enumerate}
\item each face of a triangulation $T$ of $S$ maps to a totally geodesic triangle, 
\item for each vertex $v \in T$ the angles between the images of the edges adjacent to $v $ sum to at least $2\pi$.
\end {enumerate}\end {definition}

A simplicial hyperbolic surface pulls back the hyperbolic metric on $M $ to a path-metric on $S$ that is smooth and hyperbolic away from the vertices of $T $, at which there are possible excesses of angle.  By the Gauss-Bonnet Theorem, $$\mathrm{vol}(S) \leq 2\pi (2g-2).$$  
Therefore, the injectivity radius of $S$ at every point is at most some $C=C(g)$, and the $\epsilon $-electric diameter of $S$ is at most some $C=C(\epsilon,g)$.  Accordingly,

\begin {fact}\label {boundeddiameter}
If $f: S \to M $ is a simplicial hyperbolic surface, there is a simple closed curve on $S$ whose image has length at most some $C=C(g)$.  Moreover, if $f $is $\pi_1 $-injective, the $\epsilon $-electric diameter of $f(S) \subset M$ is at most $C=C(g,\epsilon) $.
\end {fact}

Here, the $\pi_1 $-injectivity ensurers that the $\epsilon $-thin parts of $S$ map into $\epsilon $-thin parts of $M$, so that electric distance on $S $ bounds electric distance in $M$.

Most existence results about simplicial hyperbolic surfaces stem from the following, which is essentially due to Thurston.

\begin {lemma}[compare with Example 8.7.3, \cite{Thurstongeometry}]\label {existence}
Let $g : S \longrightarrow M $ be a $\pi_1$-injective map into a hyperbolic $3$-manifold $M $ without cusps and $\alpha $ a multi-curve on $S$.  Then $g $ is homotopic to a simplicial hyperbolic surface $f: S \longrightarrow M $ that \emph {realizes} $\alpha $, i.e.\ such that $f $ maps each component of $\alpha $ to a closed geodesic.
\end {lemma}

Canary showed in \cite {Canarycovering} how to interpolate between simplicial hyperbolic surfaces using simplicial hyperbolic surfaces, which gives more powerful existence results.  The following is a version of the much stronger theorem of \cite {Canarycovering} with the same name; it can be proved using the arguments of \cite[Theorem 6.2]{Canarycovering}.

\begin {theorem}[Canary's Filling Theorem]\label{filling}
There is some constant $C=C(g)$ with the following property.  Suppose that $M$ is a hyperbolic $3$-manifold homeomorphic to $S \times\BR $ and a radius $C$ ball around a point $p\in M$ is contained in the convex core of $M$. Then $p$ lies in the image of a simplicial hyperbolic surface $f: S \to M$ in the homotopy class of the fiber. 
\end {theorem}

When $\phi: S\longrightarrow S $ is pseudo-Anosov, the infinite cyclic cover $\hat M_\phi $ of the mapping torus $M_\phi $ is its own convex core (see \S\ref {cores}). So, Canary's Filling Theorem implies that every point in $M_\phi$ lies in the image of a simplicial hyperbolic surface.   Projecting to $M_\phi $, we have that every point of $M_\phi $ lies in the image of a simplicial hyperbolic surface in the homotopy class of the fiber.

Simplicial hyperbolic surfaces $f : S\longrightarrow M_\phi $ are in general not embedded.  However, work of Freedman-Hass-Scott (see \cite [Theorem 2.5] {Canarylimits}), implies that any $f $ in the homotopy class of the fiber is homotopic to an embedding, which maybe is not simplicial hyperbolic, but lies arbitrarily close to the image of $f $ and still has an electric diameter bound:

\begin {corollary}\label {levelsurfaces}
If $\phi: S\longrightarrow S$ is pseudo-Anosov, $\delta>0$, and $p\in M_\phi $, there is an embedding $S \into M$ in the homotopy class of the fiber such that
\begin {enumerate}
\item the hyperbolic distance from $p $ to $ S $ is at most $\delta $,
\item the $\epsilon $-electric diameter of $S\subset M $ is at most some $C=C(g,\epsilon) $.
\end {enumerate}
\end {corollary}

\section {Electric distance and the curve complex}
\label {circumferencesection}

As stated in Section \ref {proof}, the \emph {$\epsilon $-electric circumference} of a hyperbolic mapping torus is the least $\epsilon $-diameter of a loop projecting nontrivially to $\pi_1 S ^ 1 $, and is related to the curve complex translation distance $\tau(\phi)$ of the monodromy:

\begin {replemma}{thickcircumference}[Electric diameter and $\tau(\phi)$]
There are $\epsilon=\epsilon(g) $ and $k=k(g,\epsilon)$ such that if $M_\phi$ is the mapping torus of a pseudo-Anosov map $\phi:S \longrightarrow S,$
$$ \diameter_\epsilon (M_\phi) \approx_k \circumference_\epsilon (M_\phi )\approx_k\tau (\phi). $$
\end {replemma}
For convenience, here and below we write $A \leq_k B$ if $A \leq k \cdot B + k$ and $A \approx_k B$ if $A \leq_k B$ and $B \leq_k A$.  Lemma \ref {thickcircumference} is a consequence of the following: 

\begin {theorem}\label{models}
For every $L>0$ and small $\epsilon>0 $, there are constants $k_1=k_1(g,L)$ and $k_2=k_2(g,\epsilon) $ such that if $M \cong S\times\BR $ is a hyperbolic $3$-manifold without cusps and $\alpha,\beta$ are simple closed curves on $S$ with geodesic representatives $\hat\alpha,\hat \beta \subset M $ of length at most $L $, then 
$$d_{C (S)}  (\alpha,\beta) \, \leq_{k_1} \, d_{\epsilon} (\hat \alpha,\hat\beta)\,  \leq_{k_2} \, d_{C (S)} (\alpha,\beta).$$
\end {theorem}

This result first appeared in a paper of Bowditch \cite[Theorem 5.4]{Bowditchending}. However, we feel that credit should also be given to Yair Minsky, since Theorem \ref{models} is implicit in the development of the model manifolds of \cite{Minskyclassification}, and to Brock-Bromberg \cite{Brockgeometric}, who prove a closely related result that we will use below. We have chosen to present a proof of Theorem \ref{models} here since Bowditch's proof appears in the middle of a difficult paper and is hard to extract. 

Below, unless we say otherwise we use `bounded' to mean bounded above by a constant depending only on $\epsilon $ and on $g $.

\begin {proof}[Proof of Theorem \ref {models}]
For the upper bound, suppose that $$\alpha =\alpha_0, \alpha_1, \ldots, \alpha_n = \beta$$ is a path in $C(S)$. For each $i$, the multi-curve $\alpha_i \cup \alpha_{i + 1}$ can be realized geodesically by a simplicial hyperbolic surface $f_i:S_i \to M$ in the homotopy class of the fiber, by Lemma \ref {existence}.  By Fact \ref {boundeddiameter}, the images $f_i (S_i) $ have bounded electric diameter.
Since each image $f (S_i)$ intersects $f (S_{i+1})$ along the geodesic $\hat \alpha_{i+1}$, the electric distance from $\hat \alpha $ to $\hat\beta $ is bounded above by a linear function of $n$, and therefore a linear function of $\tau(\phi)$.

The lower bound is a bit harder.  Let $\gamma $ be a path realizing the electric distance from $\hat \alpha$ to $\hat \beta $ and subdivide $\gamma $ into a concatenation 
$$\gamma =\gamma_1 \cdots \gamma_n, $$
such that for odd $i$, the path $\gamma_i $ lies in the $\epsilon $-thick part of $M $ and for even $i $, in the $\epsilon$-thin part of $M$.  (Start the index at $i = 0 $ if necessary.)  We then have
\begin {equation}\label {lengths}\sum_{i \text { odd}} \length(\gamma_i) =  \length_\epsilon(\gamma), \end {equation}
where on the left length is absolute and on the right it is electric.  

By Canary's Filling Theorem (Theorem \ref {filling}) and Fact \ref {boundeddiameter}, for each $i $ there is a simple closed curve $\alpha_i $ on $S $ whose geodesic representative $\hat \alpha_i$ has bounded length and lies at bounded (absolute) distance from the left endpoint of $\gamma_i $.  Here and in the rest of this proof, `bounded' means by a constant depending on $g $.  As long as $\epsilon $ is much smaller than the Margulis constant $\epsilon_0 $, for even $i $ the geodesics $\hat\alpha_i$ and $\hat \alpha_{i+1} $ must lie in the same component of the $\epsilon_0 $-thin part of $M$, which can only happen if the simple closed curves $\alpha_i$ and $\alpha_{i + 1}$ are homotopic.  For odd $i$, it follows from a theorem of Brock-Bromberg \cite [Theorem 7.16]{Brockgeometric} that
$$d_{C(S)}(\alpha_i, \alpha_{i + 1}) \leq_k \length (\gamma_i), $$
where $k =k(g)$.  Therefore, we conclude from Equation \eqref{lengths} that
$$d_{C (S)} (\alpha_1,\alpha_n)\leq_k \length_\epsilon (\gamma) $$
for some slightly larger $k=k(g)$.  Applying Brock-Bromberg again to bound the distance between $\alpha,\beta $ and $\alpha_1,\alpha_n$, we can replace this with the inequality
$$d_{C (S)} (\alpha,\beta)\leq_k \length_\epsilon (\gamma), $$
except that now $k=k(g,L)$ depends on the length bound $L$ for $\hat\alpha $ and $\hat \beta $.
\end {proof}

\begin {proof}[Proof of Lemma \ref {thickcircumference}]
To set up some notation, let $\hat M_\phi\cong S\times\BR$ be the infinite cyclic cover corresponding to the subgroup $\pi_1 S\subset \pi_1 M_\phi$, see \S \ref {cores}.  We also let $$\hat\phi: \hat M_\phi \longrightarrow \hat M_\phi $$ be the isometry that generates the deck group of $\hat M_\phi \longrightarrow M_\phi $.

We first show that $\diameter_\epsilon (M_\phi) \leq_k \circumference_\epsilon (M_\phi). $  Suppose that $p\in M $ and let $\gamma$ be a loop in $M_\phi $ realizing the $\epsilon $-electric circumference.   Corollary~\ref {levelsurfaces} gives an embedded surface $S\into M_\phi $ in the homotopy class of the fiber that has bounded $\epsilon $-electric diameter and lies at bounded distance from $p $.  As $\gamma $ projects essentially under $M_\phi\longrightarrow S ^ 1 $, this surface intersects $\gamma $.  So, the $\epsilon $-electric diameter of $M_\phi $ is at most a bounded constant plus the $\epsilon $-electric diameter of $\gamma$, and therefore is at most a bounded multiple of $\tau (\phi) $.

\vspace {1mm}

Next, we show that $\circumference_\epsilon (M_\phi) \leq_k \tau(\phi). $  Pick a simple closed curve $\alpha $ on $S$ such that $d_{C(S)} (\alpha,\phi (\alpha))=\tau (\phi)$.  
Then if $\hat \alpha$ is the geodesic representative of $\alpha $ in $\hat M_\phi$, the geodesic representative of $\phi (\alpha) $ is $\hat\phi (\alpha)$. By Theorem \ref {models}, there is a path $\hat \gamma $ from $\hat \alpha$ to $ \hat \phi(\alpha) $ in $\hat M_\phi $ with
$$\length(\hat\gamma) \leq_{k_2} \tau(\phi). $$
 By Lemma \ref {existence} and Fact \ref {boundeddiameter}, $ \hat \phi(\alpha)$ has bounded $\epsilon $-electric \emph {diameter} in $\hat M_\phi $.  So, after increasing the length of $\hat\gamma $ by a bounded $\epsilon $-electric amount, we may assume that the endpoints of $\hat\gamma $ differ by $\hat\phi$.  The image of $\hat\gamma $ in $M_\phi $ is then a loop $\gamma $ that has $\epsilon $-electric length at most some multiple of $\tau(\phi)$, and which projects to an essential loop under the fibration $M_\phi \longrightarrow S ^ 1 $.   As $d_\epsilon $-length bounds $d_\epsilon $-diameter of a loop, this shows that $\circumference_\epsilon (M_\phi) \leq_k \tau(\phi) $ as desired.

\vspace {1mm}

Finally, we claim that $\tau (\phi)  \, \leq_k \, \diameter_\epsilon(M_\phi)$.   By Fact \ref {boundeddiameter} and Lemma \ref{existence}, there is a simple closed curve $\alpha $ on $S $ such  that geodesic $\hat\alpha $ in $\hat M_\phi $ has bounded length.  We claim that the $d_\epsilon (\hat \alpha,\hat \phi( \alpha))$ is at most a bounded multiple of $\diameter_\epsilon(M_\phi)$.  Theorem \ref {models} will then imply that the curve complex distance between $\alpha  $ and $\phi(\alpha) $ is similarly bounded, finishing the lemma.

Using Corollary \ref {levelsurfaces}, pick an embedding $f : S \longrightarrow M_\phi $ such that $f (S) $ has bounded $\epsilon $-electric diameter and lies at bounded distance from the projection of $\hat\alpha $ to $M_\phi $.  Then $f$ lifts to $\BZ$-many disjoint embeddings $S \longrightarrow \hat M_\phi $. If $\hat f $ is such a lift that intersects $\hat  \alpha $, the rest of the lifts are of the form ${\hat \phi}^n \circ \hat f$, where $n\in\BZ $.

As $\pi_1 M_\phi$ is generated by loops with $\epsilon $-electric length at most twice the $\epsilon $-electric diameter of $M_\phi$, there must be a loop $\gamma $ in $M_\phi $ with $$\length_\epsilon (\gamma) \leq 2\diameter_\epsilon(M_\phi)$$ that does \emph {not} lie in $\pi_1 S \subset\pi_1 M_\phi$. It follows that $\gamma $ must intersect $f(S)$.  Cutting $\gamma $ at some such intersection point, lift it to an arc $\hat\gamma $ in $\hat M_\phi$ that starts on $\hat f(S)$ and ends on $\hat\phi^n \circ \hat f (S) $ for some $n\neq 0$. Up to reversing the direction we traverse $\gamma $, we may assume $n>0$.  For $n>1$, the surface $\hat\phi \circ \hat f(S)$ separates $\hat f (S)$ from $\hat\phi ^n\circ \hat f(S)$ in $\hat M_\phi $.   Therefore, in any case $\hat\gamma$ must intersect $\hat\phi \circ \hat f (S)$.

To recap, $\hat \alpha $ lies at bounded distance from $\hat f (S)$, which has bounded electric diameter.  An arc of $\hat \gamma $ connects $\hat f (S)$ to $\hat\phi \circ \hat f(S)$, which has bounded electric diameter and intersects $\hat {\phi( \alpha)}$.  So, we would be done if we knew that the $\epsilon $-electric length of $\hat\gamma$ was bounded by a multiple of $\diameter_\epsilon (M_\phi) $.

The $\epsilon $-electric length of $\gamma$ in $  M_\phi$ is at most twice $\diameter_\epsilon(M_\phi)$.  A priori, the electric length of $\hat \gamma $ could be larger, since in $\hat M_\phi$ we electrify the $\epsilon $-thin parts of $\hat M_\phi$, which may not be the entire preimage of the $\epsilon $-thin part of $M_\phi $.  However, this does not happen, since by a result of Otal \cite{Otalgeodesiques}, any sufficiently short loop (i.e., with length less than some constant depending on genus) in a hyperbolic mapping torus is homotopic into the fiber.
\end {proof}

\section{Convexity, carrier graphs and relative length}
\label {carriergraphsection}

In this section, we give some background on convexity and carrier graphs, and introduce the notion of relative length referenced in Proposition \ref {chainsprop}. 

\vspace{1mm}

\subsection {Convex hulls in $\BH ^ n $} A subset $A \subset \BH ^n$ is called \emph {convex} if every geodesic segment with endpoints in $A $ is entirely contained in $A$.  The \emph {convex hull} of a subset $A \subset \BH ^ n $ is the smallest convex subset $hull (A) \subset\BH ^ n $ containing $A $.  Additionally, if $K \geq 0 $ then a subset $A \subset \BH ^ n$ is called \emph{$K $-quasiconvex} if every geodesic segment with endpoints in $A$ is contained in the radius-$K $ neighborhood $\CN_K (A)$ of $A $.

Every geodesic with endpoints in a set $A \subset \BH ^n$ lies inside of $hull (A) $, so any set $A $ with a $K $-neighborhood that contains $hull (A) $ is $K $-quasiconvex.  The converse is also true, up to an increase in the constant:

\begin {fact}\label {quasiconvex}
If $A \subset \BH ^ n $ is $K $-quasiconvex, then $hull (A) \subset \CN_{nK} (A) $.
\end {fact}
\begin {proof}
Let $hull ^ 1 (A) $ be the union of all geodesic segments with endpoints in $A $ and define $hull  ^ i (A) =hull ^ 1 (hull ^ {i - 1}(A))$ inductively.  As $A $ is $K $-quasiconvex, $hull ^ 1 (A) \subset \CN_K (A)$.  An inductive argument using the convexity of the distance function shows that $hull ^ i (A)  \subset \CN_{iK} (A) $ for all $i $.  But by \cite [Corollary 2.8]{Sahattchieveconvex}, for instance, in hyperbolic or Euclidean $n $-space we have $hull ^ n (A) =hull (A) $.
\end {proof}

In $\BH ^  n $, quasiconvexity is somewhat robust.

\begin {fact}\label {unions}
For every $K\geq 0 $, there is some $K' =K'(K)$ such that if $\{A_i, \, i\in I\}$ is a collection of $K $-quasiconvex subsets of $\BH ^ n$ and every $A_i $ intersects some fixed $A_{i_0}$, then the union $\cup_i A_i$ is $K' $-quasiconvex.
\end {fact}
\begin {proof}
Every geodesic with endpoints in $\cup_i A_i$ is one side of a geodesic quadrilateral whose other three sides lie in a $K $-neighborhood of $\cup_i A_i$. So, the fact follows from the (uniform) $\delta $-hyperbolicity of $\BH ^ n $.\end {proof}

Suppose that $\Gamma $ is a discrete subgroup of $\isometries (\BH ^ n) $.  The smallest nonempty $\Gamma $-invariant convex subset of $\BH ^ n $ is the convex hull of the limit set of $\Gamma $, which we  briefly introduced in \S \ref {preliminaries} and now refer to as $hull (\Gamma) $.  We will also need an enlargement of $hull(\Gamma )$ that contains all points on which $\Gamma$ acts with small displacement:  if $\epsilon>0 $ is smaller than the Margulis constant, define $$Thin_\epsilon (\Gamma) =\{x\in\BH ^n\ | \ \exists \, \gamma\in\Gamma \text { with } d (x, \gamma (x)) <\epsilon\} $$
and define the \emph {$\epsilon $-thickened convex hull} of $\Gamma$ to be the smallest nonempty $\Gamma$-invariant convex subset $hull_\epsilon (\Gamma) \subset \BH ^n $ that contains $Thin_\epsilon (\Gamma) .$ 

By $\Gamma $-invariance, the $\epsilon $-thickened convex hull contains the convex hull of $\Lambda (\Gamma) $; it also obviously contains $Thin_\epsilon (\Gamma) .$  Here is a coarse converse:

\begin {lemma}\label {thickconvexcore}
There is some universal $K$ such that $hull_\epsilon (\Gamma) $ is contained in the $K $-neighborhood of the union of $hull (\Gamma) $ and $Thin_\epsilon (\Gamma) .$
\end {lemma}

By the Margulis lemma, the components of $Thin_\epsilon (\Gamma) $ are all convex and all intersect $hull (\Gamma) $ along their core axes.  So, Lemma \ref {thickconvexcore} follows from Fact \ref {unions}.

\subsection {The geometry of (equivariant) carrier graphs}

We introduced carrier graphs in \S \ref {proof} as $\pi_1 $-surjective maps $f: X\longrightarrow M $, where $X $ is a finite graph and $M $ is a hyperbolic $n $-manifold.   Many arguments using carrier graphs are best performed in the universal cover, so we now introduce an equivariant version of carrier graphs in $\BH ^ n $. 

Suppose that $\Gamma  $ is a discrete subgroup of $\isometries (\BH ^ n) $. A \emph {$(\Delta, \Gamma) $-equivariant carrier graph} is a $(\Delta, \Gamma) $-equivariant map $$\tilde f :  \tilde X \longrightarrow \BH ^ n,$$ where $ \tilde X $ is a connected, locally finite graph and $\Delta < \mathrm {Aut} ( \tilde X)$ acts freely and cocompactly on $ \tilde X $.  For example, any carrier graph $ X \longrightarrow M $ for a hyperbolic $n$-manifold $M =\BH ^n / \Gamma $ lifts to an equivariant carrier graph for $\Gamma \circlearrowright \BH ^ n $.  Conversely, if $\tilde f : \tilde X \longrightarrow \BH ^ n $ is $(\Delta, \Gamma)$-equivariant, the quotient $\tilde X / \Delta$ is a finite graph and $\tilde f $ descends to a carrier graph $f: X / \Delta\longrightarrow \BH ^ n/\Gamma $.  So, there is a correspondence between carrier graphs and their equivariant analogues.

We now develop some basic relationships between a $(\Delta, \Gamma) $-equivariant carrier graph and the geometry of the action of $\Gamma $ on $\BH ^ n $.  All the following results can be interpreted within the quotient $\BH ^ n/\Gamma $, but they are most natural to prove in the equivariant setting.

\vspace {1mm}

First, when edges are sent to geodesic segments, the $\delta $-hyperbolicity of $\BH ^ n $ forces an equivariant carrier graph to pass near $hull_\epsilon (\Gamma) $.

\begin {lemma}\label {lemmaclose}
Suppose that $\tilde f: \tilde X \longrightarrow  \BH ^ n$ is a $(\Delta, \Gamma) $-equivariant carrier graph that maps each edge of $\tilde X$ to a geodesic segment in $\BH ^ n $, that $\alpha $ is an edge-path in $\tilde X $ from $x$ to $y$, and that $\gamma  (\tilde f (x)) =\tilde f (y)$ for some nontrivial $\gamma\in \Gamma $.  

Then there is a point on $\tilde f(\alpha) $ whose distance to $hull_\epsilon (\Gamma)$ is at most some constant depending only on $\epsilon $ and the number of edges in $\alpha $.
\end {lemma}
\begin {proof}
Let $x', y' $ be the nearest point projections of $\tilde f (x),\tilde f (y) $ to $hull_\epsilon (\Gamma) $.  Form a geodesic polygon in $\BH ^ n $ by concatenating $\tilde f(\alpha) $ with the segments $$[\tilde f (y), y'], \ [y', x'] , \ [x',\tilde f (x)] .$$  If $k $ is the number of edges in $\alpha $, then by $\delta $-hyperbolicity, $[x',\tilde f (x)] $ is contained in a $C $-neighborhood of the other sides, for some $C = C (k) $.

The angle between  $[x',\tilde f (x)] $ and $[x',y']$ is at least $\pi/2 $.  So, $\CN_C([x',y'])$ can only contain a segment of $[x',\tilde f (x)] $ with length at most some $C' = C' (k) $.  

If $\tilde f(x) \in hull_\epsilon (\Gamma) $, we're done.  Otherwise, $x'\in\partial hull_\epsilon(\Gamma) $.  As $y'=\gamma (x') $ for some nontrivial $\gamma\in \Gamma $, the segment $[y',x'] $ has length at least $\epsilon $.  So, as both $[\tilde f (x), x'] $ and $[\tilde f (y), y'] $ intersect $[y',x']$ at an angle at least $\pi/2  $, $\ \CN_C ([\tilde f (y), y'] )$ can only contain a segment of $[x',\tilde f (x)] $ of length at most $C'' = C''(k,\epsilon) $.  

Therefore, there is a point on $[x',\tilde f (x)] $ at most $\max(C',C'') $ away from $x'$ that is within $C $ of $\tilde f (Y_i) $, which bounds the distance between $\tilde f (Y_i) $ and $hull_\epsilon (\Gamma) $ by a constant depending only on $k,\epsilon $.
\end {proof}

The \emph {$\epsilon $-thickened convex hull of $\tilde f (\tilde X) $} with respect to $\Gamma $, written $$hull_\epsilon (\tilde f (\tilde X),  \Gamma) ,$$ is the convex hull of the union of $\tilde f (\tilde X) $ and $hull_\epsilon (\Gamma) .$  Actually, the union is quasiconvex, so taking the convex hull does not change it dramatically:

\begin {proposition}\label {nonewshit}
Suppose that $\tilde f: \tilde X \longrightarrow  \BH ^ n$ is a  $(\Delta, \Gamma) $-equivariant carrier graph that maps each edge of $\tilde X$ to a geodesic segment in $\BH ^ n $.  Then there is some $K $ depending only on $n $, $\epsilon $ and the number $k $ of edges in $\tilde X / \Delta$ such that
$$hull_\epsilon (\Gamma) \cup \tilde f (\tilde X) $$ 
is $K $-quasiconvex.  Consequently,
$$hull_\epsilon (\Gamma) \cup \tilde f (\tilde X) \  \subset \ hull_\epsilon (\tilde f (\tilde X), \Gamma) \ \subset \ \CN_K \left(hull_\epsilon (\Gamma) \cup \tilde f (\tilde X) \right).$$ 
\end {proposition}

\begin {proof}
The point is to show that $hull_\epsilon (\Gamma) \cup \tilde f (\tilde X) $ is $K/n $-quasiconvex, for then the second statement follows from Fact \ref {quasiconvex}.

Let $Y_i \subset\tilde X$, where $i = 1, 2, \ldots $, be all the translates of a connected fundamental domain $Y_0 $ for the cocompact action $\Delta\circlearrowright \tilde X $.  Each $\tilde f (Y_i) $ is a connected union of $k $ geodesic segments in $\BH ^ n $, and therefore is $C$-quasiconvex for some $C =C(k,n)$, by the $\delta $-hyperbolicity of $\BH ^ n $.  Applying Lemma \ref {lemmaclose} to a simple path connecting (extremal) vertices of $Y_i $ that are identified by the $\Delta $-action, we see that a $C'$-neighborhood of each $\tilde f (Y_i) $ intersects $hull_\epsilon(\Gamma) $, where $C'=C'(n,k,\epsilon) $.

Applying Fact \ref {unions} to $hull_\epsilon (\Gamma) $ and these $C' $-neighborhoods then shows that $hull_\epsilon (\Gamma) \cup \tilde f (\tilde X) $ is $K $-quasiconvex for some $K =K(k,n,\epsilon) $.
\end {proof}

\subsection {Edge moves and carrier graphs of minimal length}
\label{edgemovesection}

Let $f : X \longrightarrow M$ be a carrier graph in a hyperbolic $n$-manifold $M $.  As in Figure~\ref {edgemove}, an \emph {edge move} on $f$ creates a new carrier graph $g : Y \longrightarrow M $ as follows:

\begin {figure}[h]
\centering
\includegraphics {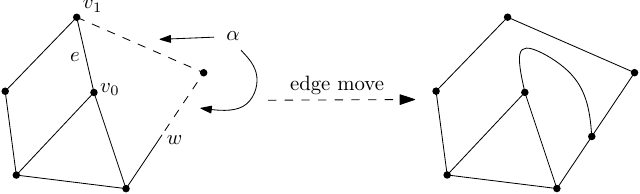}
\caption {An \emph {edge move} drags one endpoint of $e$ along the path $\alpha $.}
\label {edgemove}
\end {figure}

\begin {enumerate}[rightmargin=1cm] 
\item Pick an edge $e \subset X$ with vertices $v_0,v_1$ and a path $\alpha $ in $X\setminus e $ from $v_1$ to some point $w \in X$.  
\item Create a new graph $Y $: sever the attachment of $e$ to $v_1$, create an additional vertex at $w$ and reattach the free endpoint of $e$ to $w$.  
\item Define $g : Y\longrightarrow M $: set $g = f$ on $Y \setminus e$,  and send $e$ to a path from $f(v_0) $ to $f(w)$ in the homotopy class of $f(e) \cdot f(\alpha) $.
\item Homotope $g$ while fixing all the vertices of $Y$, except possibly $w$.
\end {enumerate}
Edge moves have a natural interpretation in the equivariant setting as well: if $\tilde f : \tilde X \longrightarrow M$ is a $(\Delta, \Gamma)$-equivariant carrier graph, an edge move on $\tilde f$ slides one endpoint of an edge $e \subset \tilde X $ along a path in $\tilde X \setminus \Delta e$ and then repeats this equivariantly for every edge in the $\Delta$-orbit $\Delta e$.

\begin {definition} We say that a carrier graph (equivariant or not) has \emph {minimal length rel vertices} if it does not admit a length-decreasing edge move.
\end {definition}

This version of minimality is strictly weaker than that used in \cite{Biringergeometry, Biringerfiniteness}, in which two carrier graphs $f_i : X_i\longrightarrow M $, $i=1,2$ were said to be \emph {equivalent} if there is a homotopy equivalence $h :X_1 \to X_2 $ such that $f_1 \simeq f_2 \circ h$, and \emph {minimal length} meant minimal in an equivalence class.  The advantage of the new definition is that it is inherited by subgraphs, since any length-reducing edge move on a subgraph clearly extends to such a move on the supergraph.  Note that in contrast to `equivalence', edge moves are not always reversible.

Most of the key geometric properties of the minimal length carrier graphs of \cite{Biringergeometry, Biringerfiniteness,Whiteinjectivity} still apply to carrier graphs with minimal length rel vertices.  In particular, we have the following adaptation of a result of White:

\begin{proposition}[see \cite{Whiteinjectivity}]
\label{minimallengths}
Let $f : X \to M$ be a carrier graph with minimal length rel vertices in a hyperbolic $n$-manifold $M$.  Then $X$ is at most trivalent and each edge maps to a geodesic segment in $M$.  The angle between any two adjacent edges is at least $\frac{2\pi}{3}$, and the image of any simple closed path in $X$ is either a point or an essential loop in $M$.
\end{proposition}

In particular, if $X$ has no valence $1$ or $2$ vertices, it is trivalent with all angles equal to $\frac{2\pi}3$ and has $3(\rank(\pi_1 X)-1)$ edges, by an Euler characteristic count.

{\begin {figure}[h]
\centering
\includegraphics {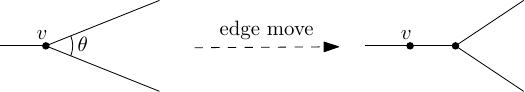}
\caption {As long as the angle $\theta < 2\pi/3$, folding part of the two edges together is a length-decreasing edge move.}
\label {angles}
\end {figure}}

To prove the proposition, if two adjacent edges meet at an angle $\theta <\frac {2\pi}{3}$, then the edge move shown in Figure~\ref{angles} decreases length.  (Since it suffices to perform the move arbitrarily close to the vertex, this can be checked with a quick Euclidean computation.)  It follows that $X$ is at most trivalent, and at every trivalent vertex the angles are all equal to $\frac{2\pi}{3}$.  Finally, if a simple closed path $\alpha$ in $X$ maps homotopically trivially to $M $, then either $\alpha $ maps to a point or there is a length-decreasing edge move that replaces some edge $e \subset \alpha$ with a loop based at one of its vertices that maps to a point in $M$.

All the topology of a hyperbolic $n $-manifold is contained in its convex core, outside of which its geometry exponentially expands.  Therefore, it should be most length-efficient for a carrier graph to stay near the convex core:

\begin {proposition}[Minimal length graphs orbit the convex core]\label {stayclose}
Suppose that $f: X\longrightarrow M$ is a reduced carrier graph with minimal length rel vertices in a hyperbolic $n $-manifold $M $.  If $core(X)$ is the Stallings core of $X$, then we have
$$ f (core (X) )\subset \CN_K (core_\epsilon (  M)), $$
where $K $ is a constant depending on $n $, $\epsilon $ and the number of edges in $core (X) $.
\end {proposition}

A carrier graph $f: X\longrightarrow M$ is \emph {reduced} if no simple closed curve on $X$ is mapped to a point in $M $.  Any carrier graph can be reduced by collapsing the problematic loops in $X$, so the above still has content for nonreduced graphs. 

The \emph{Stallings core} of $X$ is the smallest subgraph of $X $ whose inclusion is a homotopy equivalence, which is obtained by removing each edge from $X $ whose interior divides $X $ into two components, one of which is a tree.  

Finally, $core_\epsilon (  M) $ is the \emph{$\epsilon$-thickened convex core} of $M $, which is defined as the smallest convex subset of $  M$ that contains the $\epsilon $-thin part of $  M $ and whose inclusion into $  M $ is a homotopy equivalence.  If $  M =\BH ^  n/\Gamma$, then $core_\epsilon (  M)$ is exactly the projection to $  M$ of the $\epsilon $-thickened convex hull of $\Gamma $.

\begin {proof} [Proof of Proposition \ref {stayclose}]
In order to make the argument more readable, we'll omit reference to specific constants and understand that `bounded' means in terms of $n,\epsilon$ and the number of edges in $core (X)$, and that `large' and `far' mean much larger and further than all the bounded constants appearing below.

Suppose that $  f (core (X) )$ strays very far from $core_\epsilon (  M)$.  As there are only so many edges, some subsegment $e' $ of an edge $e \subset X $ must have endpoints with radically different distances to $core_\epsilon (  M)$, such that both distances are very large.  

Write $M=\BH ^ n/\Gamma $, lift $f |_{core (X)}$ to a $(\Delta,\Gamma) $-equivariant carrier graph $$\tilde f :  \widetilde{core (X)} \longrightarrow \BH ^ n, $$ choose lifts $\tilde e' \subset \tilde e \subset \tilde X $ of $e' \subset e\subset X$, let $\tilde v_0$ and $\tilde v_1$ be the endpoints of $\tilde e'$, and assume that $\tilde v_1$ lies further from $hull_\epsilon (\Gamma) $ than $\tilde v_0$.  

We claim there is a path $\tilde \alpha : [0, 1] \longrightarrow \tilde X $ with $\tilde \alpha (0) =\tilde v_1$ such that 
\begin {enumerate}
\item $\tilde \alpha(1)$ lies at bounded distance from $hull_\epsilon (\Gamma)$, 
\item the projection of $\tilde e' \cdot \tilde \alpha$ to $X$ is a simple path.
\end {enumerate} 
By Lemma \ref{lemmaclose}, any simple closed curve in $X$ must come within a bounded distance of $core_\epsilon (M)$.  If $e$ does not separate $X$, we can take $\tilde \alpha$ to be some segment of a lift of a simple closed curve in $X$ that contains $e$.  In the other case, $e$ separates $X$ into two components; let $Y $ be the component adjacent to $v_1$, the projection of $\tilde v_1 $.  As $e $ is in $core (X)$, there is a simple closed curve $\beta $ in $Y$, and we can take $\tilde \alpha $ to be the lift of a simple path in $Y $ from $v_1$  to a point on $\beta$ that lies at bounded distance from $core_\epsilon (M) $.

\begin {figure}
\centering
\includegraphics {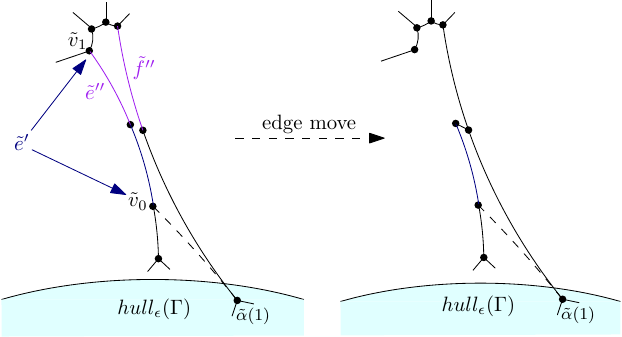}
\caption {The two segments $\tilde e''$ and $\tilde f''$ are very long and boundedly fellow travel, so the above edge move is length-decreasing.}
\label {coremove}
\end {figure}
Form a geodesic polygon in $\BH ^ n$ by concatenating 
$\tilde e' $, $ \tilde \alpha$ and the geodesic segment $[\tilde\alpha (1), \tilde v_0] $.  Divide $\tilde e'$ in half, and let $\tilde e''$ be the half adjacent to $\tilde v_1$.  By $\delta $-hyperbolicity, there is some subsegment $\tilde f''$ of an edge of $\tilde \alpha \cup [\tilde\alpha (1), \tilde v_0] $ that boundedly fellow travels a very long segment of $\tilde e''$.  By convexity, 
$$d_{\BH ^ n}\big(x,hull_\epsilon(\Gamma)\big) \leq d _{\BH ^ n}\big(\tilde v_0,hull_\epsilon(\Gamma) \big),\ \  \forall x\in [\tilde\alpha (1), \tilde v_0]. $$  So, as $\tilde e''$ lies much further from $hull_\epsilon(\Gamma) $ than $\tilde v_0$, it cannot be that $\tilde f''\subset [\tilde{\alpha}(1), \tilde{v}_0]$.  Therefore, $\tilde f''$ is part of an edge of $ \tilde X$.  As illustrated in Figure~\ref {coremove}, there is then an edge move on $\tilde X$ that dramatically decreases length.
\end {proof}

\label {minimallengthsection}

\subsection {Relative length}\label {relativelengthsection}
We now introduce the notion of relative length referenced in \S\ref {proof}, first in the equivariant setting and then for carrier graphs in $3$-manifolds.  

Suppose that $\tilde f : \tilde X \longrightarrow \BH ^ n $ is a $(\Delta, \Gamma) $-equivariant carrier graph with geodesic edges, where $\Gamma < \isometries (\BH ^ n) $.  If $\tilde Y $ is a connected subgraph of $ \tilde X $ that is stabilized cocompactly by some subgroup of $\Delta$, then $\tilde f$ restricts to an equivariant carrier graph $\tilde Y \longrightarrow \BH ^ n$ for the subgroup $\Gamma_{\tilde Y} < \Gamma $ that stabilizes $\tilde f (\tilde Y) $.  

\begin {definition}[Equivariant relative edge length] If a vertex of an edge $\tilde e$ of $\tilde X$ lies in $\tilde Y$, we define the \emph{length of $\tilde e $ relative to $\tilde Y$ and $\epsilon $} to be the hyperbolic length of the part of $\tilde e $ that lies outside of a radius-$1$ neighborhood of $hull_\epsilon (\tilde f (\tilde Y) ,\Gamma_{\tilde Y}) $.

Moreover, if $\tilde Y_0 $ and $\tilde Y_1 $ are both subgraphs as above and the two vertices of $\tilde e$ lie in $\tilde Y_0 $ and $\tilde Y_1 $, respectively, we define the \emph {length of $\tilde e $ relative to $\tilde Y_0,\tilde Y_1$ and $\epsilon $} as the hyperbolic length of the part of $\tilde e$ the lies outside the radius-$1$ neighborhoods of both of the associated $\epsilon $-thickened convex hulls.
\end {definition}

Note that the length of an edge of $\tilde Y $ relative to $\tilde Y$ is zero. However, we are really interested in the case that $\tilde e $ is an edge of $\tilde X \setminus \tilde Y$ but has a vertex in $\tilde Y$. 

Let's extend this definition to the non-equivariant case.  Suppose that $$f : X \longrightarrow M $$ is a carrier graph with geodesic edges in a hyperbolic $n$-manifold $M $ and that $Y\subset X $ is some subgraph.  If $e$ is an edge of $X $ that is adjacent to $Y $, lift $e $ to an edge $\tilde e$ in an equivariant carrier graph $\tilde  f: \tilde X \longrightarrow \BH ^ n $, where $\pi: \tilde X \longrightarrow X $ is the universal cover.  We define the \emph{length of $e$ relative to $ Y$
 and $\epsilon $}, written $$\length_{Y,\epsilon} (e) $$ as the length of $\tilde e$ relative to $\epsilon $ and any components of $\pi ^ {- 1}(Y) $ that $\tilde e$ touches.  The \emph {relative length of a subgraph $Z \subset X$}, written $\length_{Y,\epsilon} (Z)$, is then defined to be the sum of the relative lengths of its edges.

This definition is a slight extension of that given in \cite {Biringergeometry}, where it was always assumed that $\epsilon = 1 $.  This restriction does make a real difference in relative length, not just a difference that is bounded by a function of $\epsilon $.  However, as you might expect, the choice $\epsilon = 1 $ was arbitrary and the proof from \cite{Biringergeometry} of the following proposition, previously stated in \S \ref {proof}, goes through unchanged.

\begin{repproposition}{chainsprop}[Chains of Bounded Length, \cite {Biringergeometry}] 
Fixing $\epsilon>0 $, if $M$ is a closed hyperbolic $n$-manifold such that $\pi_1 M $ does not split as a free product, e.g.\ if $M $ is closed, and $f: X \to M$ is a carrier graph with minimal length rel vertices, there is a sequence of (possibly disconnected) subgraphs $$\emptyset = Y_0 \subset Y_1 \subset \ldots \subset Y_k = X$$ such that $\length_{Y_i,\epsilon}(Y_{i+1}) < L$ for all $i$, where $L=L(\epsilon,n) $.
\end{repproposition}

Note the assumption in Proposition \ref {chainsprop} that $\pi_1 M $ does not split as a free product.  Without this, $M$ could be the quotient of $\BH ^ n $ by a Schottky group generated by two hyperbolic-type isometries that both have large translation distance and whose axes are far apart. In this case, a minimal length, rank 2 carrier graph in $M$ might look like the union of the two closed geodesics in $M $ corresponding to the generators and a geodesic segment connecting them, modified so that all angles are $2\pi/3$.  In particular, there may not be any short edges at all.  The basic idea of the proof is to show inductively that if there are no edges in $X \setminus Y_i$ with short relative length, then $\pi_1 M$ splits as a free product of the images of the fundamental groups of the components of $Y_i $, which is impossible if $M $ is closed.

\begin {remark}
The reader may wonder why the `radius $1$-neighborhood' is needed in the definition of relative length. As in  \cite {Biringergeometry}, the $1$ here is arbitrary, but a different choice will affect the constants in Proposition~\ref {chainsprop}. The issue is that an edge of $Y_{i+1}$ could exit the convex hull of $Y_i$ at a very small angle, and then stay close for a long time. After exiting a neighborhood of definite radius, however, the remaining length will be bounded.
\end {remark}

\section {The proof of Lemma \ref {hullbounds}}
\label {lemmasection}
To remove a level of notational complexity, we assume in this section that our carrier graphs are embedded.  This does not simplify any of the proofs, but allows us to omit reference to $f$ and $\tilde f$, which makes the arguments more readable.  

We also encourage the reader to read through \S\ref {carriergraphsection} before continuing, since we will make heavy use of notation introduced there.

\vspace{1mm}

Let's recall the setup from Section \ref {proof}.  Fix a homeomorphism $\phi : S \longrightarrow S$ of a closed surface of genus $g$, and suppose that $X \hookrightarrow M_\phi$ is a minimal length carrier graph with $\rank (\pi_1 X) =m$.  From Proposition \ref {chainsprop}, we have a sequence  $$\emptyset = Y_0 \subset Y_1 \subset \ldots \subset Y_k = X$$ such that $\length_{Y_i,\epsilon}(Y_{i+1}) < L$ for all $i$, where $L=L(\epsilon,m)$ is some constant. Let $d_\epsilon $ be the pseudometric on $M_\phi $ obtained by electrifying its $\epsilon $-thin parts.

 The remaining task is to prove Lemma \ref {hullbounds}, rephrased with $X $ embedded.

\begin {replemma}{hullbounds}
Fixing $i$, suppose that for each connected component $Z \subset Y_i $, the inclusion-induced image of $\pi_1 Z $ is an infinite index subgroup of $\pi_1 S \lhd \pi_1 M_\phi $.  

Then the $d_{\epsilon} $-diameter of every component of $Y_{i+1}$ in $M_\phi $ is bounded above by a constant depending on $g,m,\epsilon $ and the $d_{\epsilon} $-diameters of the components of $Y_i $.
\end {replemma}
\begin {proof}
We may assume that $X $ does not have any valence $1$ or $2$ vertices.  By Proposition \ref {minimallengths}, $X$ is trivalent.  As $\rank(\pi_1 X) =m$, the graph $X $ has $3m-3$ edges.  So, it suffices to bound the $d_\epsilon $-diameter of each edge $e$ of $Y_{i+1}$ individually.  

We know that the length of $e$ is bounded relative to $Y_i$ and $\epsilon $.  The parts of $e$ that do not count towards relative length lie in (radius-$1$ neighborhoods of) the projections to $M_\phi$ of  certain convex sets in $\BH^3$.
So, to bound the $d_\epsilon $-diameter of $e $, it suffices to bound the $d_\epsilon $-diameters of these projections.

Fix a universal cover $\BH^3\longrightarrow M_\phi$, let $\tilde Z $ be a component of the preimage of $Y_i $ and let $\Gamma$ be the group of all deck transformations that stabilize $\tilde Z $, so that $\tilde Z/\Gamma $ is identified with a component $Z \subset Y_i $.  By Proposition~\ref {nonewshit} and Lemma~\ref {thickconvexcore}, 
$$(*) \ \ \ \ \ \ hull_\epsilon (\tilde Z, \Gamma) \ \subset \ \CN_K \big(hull(\Gamma) \cup Thin_\epsilon (\Gamma) \cup \tilde Z\big), \ \ \ \ \ \ \ $$
for some constant $K =K(\epsilon,m)$, where the dependence on $m$ is since Proposition \ref {nonewshit} requires a bound on the number of edges of $Z $, which is at most $3m-3 $.

$Thin_\epsilon (\Gamma) $ projects into the $\epsilon $-thin part of $M_\phi$, so its $d_\epsilon $-diameter is zero.   The constant in the statement of the lemma is \emph {supposed} to depend on the $d_\epsilon $-diameter of the components of $Y_i$, one of which is the projection $Z$ of $\tilde Z $ to $M_\phi $.  So, as the union in $(*)$ is coarsely connected (Proposition \ref {nonewshit}) we have reduced the lemma to bounding the $d_\epsilon $-diameter of the projection of $hull (\Gamma) $ to $M_\phi $.

Consider the cover $N =\BH ^ 3/\Gamma$ of $M_\phi $.  As the group $\Gamma $ is isomorphic to the image of $\pi_1 (Z)$ in $\pi_1 S \lhd \pi_1 M_\phi $, it is an infinite index subgroup of a surface group, and therefore is free.  By the Tameness Theorem \cite {Agoltameness,Calegarishrinkwrapping}, $N$  is homeomorphic to the interior of a handlebody.   Canary's covering theorem \cite {Canarycovering} implies that a one ended $3$-manifold whose end is degenerate cannot cover a closed manifold.  So, the convex core $core(N)$, which is the projection of $hull (\Gamma) $ to $N $, is compact.  

Let $\pi: N\longrightarrow M_\phi $ be the covering map.  We want to show that $\pi (core(N))$ has $d_\epsilon $-diameter bounded by a constant depending on $m,\epsilon $ and $diam_\epsilon(Z)$.

\begin {figure}[t]
\centering
\includegraphics {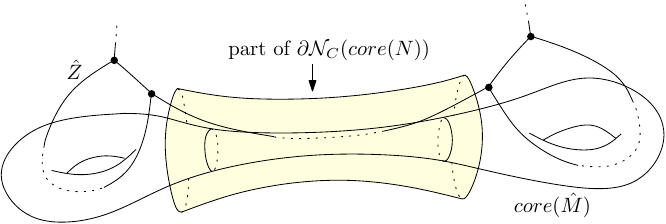}
\caption {The carrier graph $\hat Z$ is forced to track any compressible $\epsilon $-thin parts of $\partial core (N) $.}
\label {core}
\end {figure}
We first show that the projection $\pi (\partial core(N)) $ of its \emph{boundary} has bounded $d_\epsilon $-diameter.  By work of Thurston \cite {Thurstongeometry}, the intrinsic metric on $\partial core(N)$ is hyperbolic.  The components of the $\epsilon $-thick part of $\partial core(N)$ then have bounded diameter, so the same is true of their projections in $M_\phi $.  The $\epsilon $-thin parts of $\partial core (N) $ come in two flavors: either the core curve is essential in $N$ or it is not.  The map $\pi $ sends thin parts of the first type into the $\epsilon $-thin part of $M_\phi $, so their projections have zero $d_\epsilon $-diameter.  Therefore, to bound the $d_\epsilon$-diameter of $\pi (\partial core (N) )$ it suffices to deal with the compressible $\epsilon $-thin parts.

The graph $Z$ lifts homeomorphically to a carrier graph $\hat Z \into N$ with minimal length rel vertices.  By Proposition \ref {stayclose}, we have
$$core (\hat Z) \subset \CN_C (core (N)),$$
where $core (\hat Z) $ is the Stallings core of $\hat Z$ and $C=C(\epsilon ,m) $. 
Let $A \subset \partial core (N)$ be a compressible $\epsilon $-thin part and let $p\in A$.  There is a properly embedded, bounded diameter disk $D \subset \CN_C (core (N))$ with $p \in D$:  the isoperimetric inequality gives such a disk in $ core (N)$ whose boundary passes through $p$, and this can be extended by an annulus from $ \partial core (N)$ to $\partial \CN_C (core (N))$.  As the inclusion $\hat Z \into \CN_C (core (N))$ is $\pi_1 $-surjective, $\hat Z$ must intersect $D$, and therefore passes within bounded distance of $p $ (see Figure~\ref {core}).  Thus, the $d_\epsilon $-diameter of the projection $\pi (A )$ is only boundedly bigger than the $d_\epsilon $-diameter of $Z $.

We now know that $\pi (\partial core(N))$ has bounded $d_\epsilon $-diameter, and we would like to say the same about $\pi (core(N))$.  To do this, we prove:

\begin {claim}
If $f : S \longrightarrow M_\phi$ is an embedding in the homotopy class of the fiber,
$$f(S)\cap \pi (core(N))\neq \emptyset \implies f(S) \cap \pi (\partial core(N))\neq \emptyset.$$
\end {claim}
\begin {proof}
Lift both $\pi (core(N))$ and $f(S)$ to the infinite cyclic cover $\hat M_\phi $ corresponding to $\pi_1 S \lhd \pi_1 M_\phi$, making sure to preserve the intersection. The frontier of the lift of $\pi (core(N))$ is connected, since it is the lift of $\pi (\partial core(N))$. As the lift of $f(S)$ separates $\hat M_\phi $, it must also intersect this frontier.
\end {proof}

By Corollary \ref {levelsurfaces}, for any point $p\in M_\phi$, there is an $f$ whose image $f(S)$ comes arbitrarily close to $p$, and has bounded $d_\epsilon $-diameter. By this and the claim, $\pi (core(N))$ lies in a bounded $d_\epsilon$-neighborhood of $\pi (\partial core(N))$, so the desired $d_\epsilon$-diameter bound transfers from 
 $\pi (\partial core(N))$ to $\pi (core(N))$.
\end {proof}

\textit{\textrm{
\bibliographystyle{amsplain}
\bibliography{biblio}
}}

\end{document}